\documentclass[runningheads]{llncs}

\usepackage[T1]{fontenc}
\usepackage[utf8]{inputenc}
\usepackage[russian,english]{babel}

\usepackage{graphicx}
\usepackage{amssymb}
\usepackage{amsmath}

\usepackage{booktabs}
\usepackage{tabularx}
\usepackage{makecell}

\usepackage{hyperref}

\usepackage{algorithm}
\usepackage{algpseudocode}

\begin{document}
	
	\title{Frank-Wolfe Algorithms for $(L_0, L_1)$-smooth functions}
	
	\author{
		A.A. Vyguzov\inst{1,3,4}\orcidID{0009-0005-1681-1750} \and
		F.S. Stonyakin\inst{1,2,3}\orcidID{0000-0002-9250-4438}
	}
	
	\authorrunning{A. Vyguzov et al.}
	
	\institute{
		Moscow Institute of Physics and Technology, Dolgoprudny, Institutsky lane, 9, Russia \and
		Simferopol, Academician Vernadsky Avenue, 4, V.~I.~Vernadsky Crimean Federal University, Republic of Crimea, Russia \and
		Innopolis University, Kazan, Tatarstan, 420500, Russia \and
		Adyghe State University, Maikop, Russia, Pervomayskaya Str., 208
	}
	
	\maketitle
	
	\begin{abstract}
		We propose a new version of the Frank--Wolfe method, called the \emph{$(L_0,L_1)$-Frank--Wolfe} algorithm, developed for optimization problems with $(L_0,L_1)$-smooth objectives. We establish that this algorithm achieves superior theoretical convergence rates compared to the classical Frank--Wolfe method. In addition, we introduce a novel adaptive procedure, termed the \emph{Adaptive $(L_0,L_1)$-Frank--Wolfe} algorithm, which dynamically adjusts the smoothness parameters to further improve performance and stability. Comprehensive numerical experiments\footnote{\href{https://github.com/DredderGun/l0_l1_smooth_amd_fw_article}{The code can be downloaded here}} confirm the theoretical results and demonstrate the clear practical advantages of both proposed algorithms over existing Frank--Wolfe variants.
		
	\end{abstract}
	
	\section{Introduction}
	
	There are many methods for constrained optimization: the projected gradient method~\cite{polyak1983introduction}, barrier function methods, and penalty function methods~\cite{boyd2004convex,nesterov2018lectures}, among others. However, the Frank--Wolfe method, initially proposed in the seminal work~\cite{frank1956algorithm} and further generalized in~\cite{levitin1966constrained}, has recently gained significant popularity due to its inexpensive iteration cost compared to the projected gradient method (see~\cite{braun2022conditional}, Table~1.1 and~\cite{combettes2021complexity}). Moreover, the Frank--Wolfe method possesses the useful ability to generate sparse solutions, which is beneficial in many applications.
	
	In modern optimization, the Frank--Wolfe method is quite well studied (see the reviews~\cite{bomze2021frank,braun2022conditional}). In the seminal work~\cite{levitin1966constrained}, it was proved that the optimal convergence rate of this method is $O(1/k)$ for convex functions; moreover, a series of linear convergence results for the classical Frank--Wolfe method has been established~\cite{bomze2021frank,braun2022conditional}. However, these results mostly rely on the standard $L$-smoothness assumption, which is a rather restrictive class of functions. To relax this assumption, attempts have been made to introduce relatively smooth objectives~\cite{vyguzov2025adaptive,takahashi2025fast} for the Frank--Wolfe method, leading to a broader function class—yet many important modern machine learning problems remain uncovered.
	
	Recently,~\cite{zhang2019gradient} experimentally observed that modern language modeling tasks satisfy a property called $(L_0, L_1)$-smoothness, which generalizes standard $L$-smoothness. The pioneering work~\cite{zhang2019gradient} studied the clipped gradient method, and subsequently, many other gradient-based methods have been investigated for this class of functions. Although the $(L_0,L_1)$-smoothness framework has been extensively studied for gradient-based methods, very few works have addressed the Frank--Wolfe method. A related analysis of a hybrid between the clipped gradient method and the Frank--Wolfe algorithm was proposed in~\cite{pethick2025generalized}, but it did not establish linear convergence guarantees. In contrast, our work focuses on the classical (vanilla) Frank--Wolfe method for the convex case and demonstrates that exploiting the $(L_0,L_1)$-smooth structure of the objective can lead to significant acceleration benefits.
	
	There exist different step-size rules for the Frank--Wolfe method, but the most popular variants are the decreasing step size $2/(k+1)$ and the short step size (see the overviews in~\cite{braun2022conditional,bomze2021frank}). The decreasing step size achieves the optimal sublinear convergence rate and does not depend on any function parameters. Therefore, we focus on the short step size, where the $(L_0, L_1)$-smoothness parameters can be utilized to achieve improved convergence rates.
	
	Adaptive step-size strategies have recently gained attention in the Frank--Wolfe literature, with several works~\cite{aivazian2023adaptive,pedregosa2020linearly,vyguzov2025adaptive,khademi2025adaptive} proposing backtracking schemes that preserve convergence guarantees. However, these methods rely on a single $L$-smoothness parameter. We extend this framework by introducing a two-parameter adaptive procedure that simultaneously updates $L_0$ and $L_1$ based on their relative influence in the composite term $L_0 + L_1 \|\nabla f(x_k)\|$. This design enables more responsive yet stable parameter tuning.
	
	The main contributions of this work are as follows. We propose a new variant of the Frank--Wolfe algorithm tailored for $(L_0, L_1)$-smooth objectives and provide rigorous convergence guarantees. Specifically, we prove linear convergence rates, demonstrating superior performance compared to classical Frank--Wolfe. In the general convex setting, we establish that the algorithm attains the optimal $O(1/k)$ convergence rate, ensuring it is not worse than the standard method. Additionally, we introduce a novel adaptive procedure for independently updating the $(L_0, L_1)$ parameters, which yields improved empirical performance. Numerical experiments confirm that the proposed algorithm consistently outperforms both standard and existing adaptive Frank--Wolfe variants across diverse benchmarks.
	
	\section{Preliminaries}
	
	In this paper, we consider the minimization problem
	\begin{align}
		\min_{x \in Q} \; f(x),
	\end{align}
	where $Q$ is a convex compact set and $f$ is a convex and $(L_0,L_1)$-smooth function.
	
	\begin{definition}
		A function $f$ is said to be \emph{$(L_0,L_1)$-smooth} if, for all $x$, it satisfies
		\begin{equation}\label{def:l0_l1_twice_dif}
			\| \nabla^2 f(x) \| \leq L_0 + L_1 \| \nabla f(x) \|,
		\end{equation}
		for some constants $L_0, L_1 > 0$. Throughout this paper, unless specified otherwise, we use the standard Euclidean norm $\|\cdot\|$ for vectors and the spectral norm $\|\cdot\|$ for matrices.
	\end{definition}
	
	This notion was introduced in \cite{zhang2019gradient}, where it was used to explain the superior convergence behavior of the clipped gradient descent method on deep learning problems compared to standard gradient descent (GD).
	
	Later, in \cite{zhang2020improved} (see Remark 2.3), this definition was extended from twice differentiable to once differentiable functions. Specifically, a differentiable function $f$ is said to be $(L_0, L_1)$-smooth if, for all $x, y \in \mathbb{R}^d$ such that $\| x - y \| \leq \frac{1}{L_1}$, we have
	\begin{equation}\label{def:l0_l1}
		\|\nabla f(x) - \nabla f(y)\| \le (L_0 + L_1 \|\nabla f(y)\|)\|x - y\|.
	\end{equation}
	
	This condition generalizes the standard $L$-smoothness assumption. Indeed, when $L_1 = 0$, inequalities (\ref{def:l0_l1_twice_dif})–(\ref{def:l0_l1}) reduce to the classical definition of $L$-smoothness. Hence, the $(L_0, L_1)$-smoothness condition defines a broader class of functions than the conventional $L$-smooth case. 
	
	Motivated by this generalization, we modify the standard shortest-step rule in the Frank–Wolfe method. Let us briefly recall the classical Frank–Wolfe algorithm \cite{frank1956algorithm} with the shortest-step rule \cite{levitin1966constrained}. The update rule is given by
	\begin{equation}\label{fw_step_size}
		x_{k+1} = x_k + \alpha_k d_k,
	\end{equation}
	where $k$ is the iteration index, $\alpha_k \in [0, 1]$, $d_k = s_k - x_k$, and 
	\begin{equation}\label{lmo}
		s_k \in \mathrm{LMO}_Q(\nabla f(x_k)) = \arg\min_{z \in Q} \, (\nabla f(x_k)^\top z)
	\end{equation}
	is the output of the linear minimization oracle (LMO) over $Q$. In the case of the shortest-step rule, the step size is chosen as
	\begin{equation}\label{shortest_step_def}
		\alpha_k := \min \left\{ 1, \frac{-\nabla f(x_k)^\top d_k}{L \| d_k \|^2} \right\}.
	\end{equation}
	
	This observation motivates us to extend the Frank–Wolfe method \eqref{fw_step_size} to $(L_0, L_1)$-smooth objectives by introducing the following step-size rule:
	\begin{equation}\label{shortest_step_def_l0l1}
		\alpha_k := \min \left\{ 1, 
		\frac{- \nabla f(x_k)^\top d_k}{(L_0 + L_1 \| \nabla f(x_k)\|) \, \|d_k\|^2 \, e} \right\}.
	\end{equation}
	
	It is well known (see, for example, \cite{braun2022conditional}, \cite{bomze2021frank}, \cite{levitin1966constrained}) that the Frank–Wolfe method achieves a linear convergence rate under additional assumptions. Our theoretical analysis shows that the proposed Frank–Wolfe variant, with the step size defined in~\eqref{shortest_step_def_l0l1}, exhibits a superior convergence rate. The acceleration arises from the fact that $L_1 < L$ and $L_0 < L$, where $L$ denotes the standard smoothness parameter. When the solution lies in the interior of a function satisfying the PL condition, our algorithm is significantly faster than the standard method due to the absence of the PL-condition parameter $\mu$ in the convergence rate when $L_0 > L_1 \| \nabla f(x_k) \|$. A summary of the corresponding convergence rate estimates is provided in Table~\ref{tab:fw_convergence_summary}.
	
	Moreover, as mentioned earlier, not all $(L_0, L_1)$-smooth functions are $L$-smooth, meaning that our algorithm applies to a broader class of functions. 
	To illustrate this, we list several examples (see \cite{gorbunov2024methods}):
	\begin{itemize}
		\item $f(x) = \|x\|^n$, where $n$ is a positive integer, with $L_0 = 2n$ and $L_1 = 2n - 1$;
		\item $f(x) = \exp(a^\top x)$, where $L_0 = 0$ and $L_1 = \|a\|$;
		\item the logistic loss $f(x) = \log(1 + \exp(-a^\top x))$, where $a \in \mathbb{R}^d$, for which $L_0 = 0$ and $L_1 = \|a\|$, while the standard smoothness constant is $L = \|a\|^2$, which is typically much larger than $L_1$.
	\end{itemize}
	
	\begin{table}[ht]
		\centering
		\caption{Convergence rates of the standard Frank–Wolfe (FW) method with the shortest step size (see, e.g., \cite{braun2022conditional}, \cite{levitin1966constrained}) and the proposed $(L_0,L_1)$-FW Algorithm~\ref{alg:fw_l0l1} in the convergence rate when under different assumptions. 
			Abbreviations: PL – Polyak–Łojasiewicz condition; C – convex; SC – strongly convex. 
			Here, $B(x,r) \overset{\text{def}}{=} \{ z: \| z - x \| \leq r \}$, $\lambda$ denotes the strong convexity constant of the feasible set (see Definition~\ref{def:strongly_conv_sets}), $\mu$ is the PL-condition constant, $T$ is the number of iterations for which $L_0 \leq L_1 \| \nabla f(x) \|$, and $K$ is the number of iterations for which $L_0 > L_1 \| \nabla f(x) \|$.}
		\label{tab:fw_convergence_summary}
		\medskip
		\begin{tabular}{|l|l|l|l|p{6cm}|}
			\hline
			\textbf{Algorithm} & \textbf{Objective} & \textbf{Domain} & \textbf{Assumptions} & \textbf{Rate} \\ \hline
			
			FW & C; $L$-smooth & SC & $\|\nabla f(x_k)\|\ge c > 0$ &
			{\scriptsize $f(x_k)-f^\ast \le (f(x_0)-f^\ast)
				\max\!\left\{\frac{1}{2},\,1-\frac{\lambda c}{2L}\right\}^{\,k-1}$} 
			\\ \hline
			
			$(L_0, L_1)$-FW \\ Algorithm~\ref{alg:fw_l0l1} & C; $(L_0,L_1)$-sm. & SC & $\|\nabla f(x_k)\|\ge c > 0$ &
			\parbox[t]{6cm}{\scriptsize
				$f(x_{k+1}) - f^* \le (f(x_0) - f^*) \times$ \\
				$\max\!\left\{ \frac{1}{2}, 1 - \frac{\lambda}{2 e L_1} \right\}^T \times$ \\
				$\max\!\left\{ \frac{1}{2}, 1 - \frac{\lambda c}{2 e L_0} \right\}^K$  (see Th.\ref{th:linear_strong_conv_set})
			}
			\\ \hline
			
			FW & C; PL; $L$-smooth & C & $B(x^\ast,r)\subseteq Q$ &
			{\scriptsize $f(x_k)-f^\ast \le (f(x_0)-f^\ast)
				\max\!\left\{\frac{1}{2},\, 1-\frac{r^2\mu}{LD^2}\right\}^{\,k-1}$}
			\\ \hline
			
			$(L_0,L_1)$-FW \\ Algorithm~\ref{alg:fw_l0l1} & C; PL; $(L_0,L_1)$-sm. & C & $B(x^\ast,r)\subseteq Q$ &
			\parbox[t]{6cm}{\scriptsize
				$f(x_{k+1}) - f^* \le (f(x_0) - f^*) \times$ \\
				$\max\!\left\{ \frac{1}{2}, 1 - \frac{r}{4 e L_1 D^2} \right\}^T \times$ \\
				$\max\!\left\{ \frac{1}{2}, 1 - \frac{r^2 \mu}{2 e L_0 D^2} \right\}^K$ (see Th.\ref{th:linear_conv_pl_cond})
			}
			\\ \hline
			
			FW & C; $L$-smooth & C &  &
			{\scriptsize $f(x_{k+1}) - f^* \le \frac{2 L D^2}{k + 3}$}
			\\ \hline
			
			$(L_0,L_1)$-FW \\ Algorithm~\ref{alg:fw_l0l1} & C; $(L_0,L_1)$-smooth & C & &
			\parbox[t]{6cm}{\scriptsize
				$f(x^k) - f(x^*) \le \frac{2e \big(L_0 + L_1 \max_k \| \nabla f(x_k) \| \big) D^2}{k + 3}$
			}
			\\ \hline
		\end{tabular}
	\end{table}
	
	In the work \cite{chen2023generalized} (Proposition 1) it was shown that $(L_0, L_1)$-smoothness implies the next inequality, which we will use as an upper bound of our target function:
	
	\begin{lemma}[Proposition 1 from \cite{chen2023generalized} Lemma 2.5]\label{lemma:upper_bnd}
		Definition \ref{def:l0_l1} holds if and only if for all $x, y \in \mathbb{R}^n$:
		\begin{equation*}
			f(y) \leq f(x) + \langle \nabla f(x), y - x \rangle 
			+ \frac{L_0 + L_1 \|\nabla f(x)\|}{2} \exp\left( L_1 \|x - y\| \right) \|x - y\|^2 .
		\end{equation*}
	\end{lemma}
	
	We will derive our algorithm from that upper bound. That inequality can be seen as a generalization of the standard quadratic upper bound \cite{nesterov2018lectures} with $L_1 = 0$.
	
	The Frank–Wolfe gap~\cite{bomze2021frank,braun2022conditional} serves as a standard optimality measure and practical stopping criterion.
	
	\begin{definition}\label{def:fw_gap}
		The Frank–Wolfe gap of a function $f: Q \to \mathbb{R}$ is defined as
		\begin{equation*}
			G(x) = \max_{s \in Q} - \nabla f(x)^\top(s - x).
		\end{equation*}
	\end{definition}
	
	which satisfies the fundamental inequality
	
	\begin{equation}\label{fw_gap}
		G(x) \geq \nabla f(x)^\top (x - x^*) \geq f(x) - f^* > 0.
	\end{equation}
	
	\section{$(L_0,L_1)$-FW Algorithm}
	
	The $(L_0,L_1)$-Frank–Wolfe (FW) algorithm is presented in Algorithm~\ref{alg:fw_l0l1}.
	
	\begin{algorithm}
		\caption{$(L_0,L_1)$-Frank–Wolfe Algorithm ($(L_0,L_1)$-FW)}
		\label{alg:fw_l0l1}
		\begin{algorithmic}[1]
			\State \textbf{Input:} fixed parameters $L_0 > 0$, $L_1 > 0$, and the maximum number of iterations $N$
			\For{$k = 0, 1, \dots, N-1$}
			\State $s_k \gets \mathrm{LMO}(\nabla f(x_k))$
			\State $d_k \gets s_k - x_k$
			\State $\displaystyle 
			\alpha_k \gets \min \left\{ 1,\;
			\frac{-\, \nabla f(x_k)^\top d_k}{\left(L_0 + L_1 \|\nabla f(x_k)\|\right) \|d_k\|^2 e} \right\}$
			\State $x_{k+1} \gets x_k + \alpha_k d_k$
			\EndFor
			\State \textbf{Output:} final iterate $x_N$
		\end{algorithmic}
	\end{algorithm}
	
	In the following sections, we derive convergence estimates for Algorithm~\ref{alg:fw_l0l1} under different assumptions.  
	Before doing so, we first establish a descent lemma, formulated for the class of $(L_0,L_1)$-smooth objective functions.
	
	\begin{lemma}[Descent Lemma]\label{lemma:descent_lemma}
		Let $f$ be a convex function satisfying the $(L_{0},L_{1})$-smoothness condition (Definition~\ref{def:l0_l1}).  
		Then, for Algorithm~\ref{alg:fw_l0l1} and every iteration $k \ge 1$, the following inequality holds:
		\begin{equation*}
			f(x_{k+1}) - f(x_{k}) 
			\le \frac{\nabla f(x_{k})^{\top} d_{k}}{2}
			\cdot 
			\min\!\left\{1,\;
			\frac{-\,\nabla f(x_{k})^{\top} d_{k}}
			{\big(L_{0}+L_{1}\|\nabla f(x_{k})\|\big)\|d_{k}\|^{2}e}
			\right\},
		\end{equation*}
		where $D = \max_{x,y \in Q}\|x-y\|$ and $d_{k}=s_{k}-x_{k}$ with 
		$s_{k}\in \mathrm{LMO}(\nabla f(x_{k}))$.
	\end{lemma}
	
	\begin{proof}
		Let $a_{k}=L_{0}+L_{1}\|\nabla f(x_{k})\|$.  
		Since $f$ satisfies the $(L_{0},L_{1})$-smoothness condition, 
		Lemma~\ref{lemma:upper_bnd} implies
		\begin{equation}\label{th:sublinear_conv_rate2:upper_bnd}
			f(x_{k+1})-f(x_{k})
			\le \alpha_{k}\nabla f(x_{k})^{\top}d_{k}
			+\frac{a_{k}}{2}e^{\alpha_{k}L_{1}\|d_{k}\|}
			\alpha_{k}^{2}\|d_{k}\|^{2}.
		\end{equation}
		If $\alpha_{k}$ is chosen so that
		\begin{equation}\label{alpha_upper_bnd}
			\|x_{k+1}-x_{k}\| = \alpha_{k}\|d_{k}\| \le \frac{1}{L_{1}},
		\end{equation}
		then $e^{\alpha_{k}L_{1}\|d_{k}\|} \le e$, and inequality
		\eqref{th:sublinear_conv_rate2:upper_bnd} simplifies to
		\begin{equation}\label{th:sublinear_conv_rate2:upper_bnd_only_e}
			f(x_{k+1})-f(x_{k})
			\le \alpha_{k}\nabla f(x_{k})^{\top}d_{k}
			+\frac{a_{k}e}{2}\alpha_{k}^{2}\|d_{k}\|^{2}.
		\end{equation}
		Minimizing the right-hand side of 
		\eqref{th:sublinear_conv_rate2:upper_bnd_only_e} 
		over $\alpha_{k}\in(0,1]$ yields
		\begin{equation}\label{step_size_def}
			\alpha_{k}^{*}
			=\min\!\left\{1,\;
			\frac{-\,\nabla f(x_{k})^{\top}d_{k}}
			{a_{k}\|d_{k}\|^{2}e}\right\}.
		\end{equation}
		
		We verify that $\alpha_{k}^{*}$ satisfies \eqref{alpha_upper_bnd}.  
		When $\alpha_{k}^{*}<1$, we have
		\[
		(-\nabla f(x_{k})^{\top}d_{k})L_{1}\|d_{k}\|
		\le \|\nabla f(x_{k})\|\,\|d_{k}\|L_{1}\|d_{k}\|e
		\le a_{k}\|d_{k}\|^{2}e,
		\]
		which implies
		\[
		\frac{-\,\nabla f(x_{k})^{\top}d_{k}}
		{a_{k}\|d_{k}\|^{2}e}
		\le \frac{1}{L_{1}\|d_{k}\|}.
		\]
		If $\alpha_{k}^{*}=1$, the same inequality follows directly:
		\[
		1\le 
		\frac{-\,\nabla f(x_{k})^{\top}d_{k}}
		{a_{k}\|d_{k}\|^{2}e}
		\le \frac{1}{L_{1}\|d_{k}\|}.
		\]
		Thus, $\alpha_{k}^{*}$ satisfies \eqref{alpha_upper_bnd}, and 
		\eqref{th:sublinear_conv_rate2:upper_bnd_only_e} 
		holds with $\alpha_{k}=\alpha_{k}^{*}$.
		
		Substituting this step size into 
		\eqref{th:sublinear_conv_rate2:upper_bnd_only_e} yields
		\begin{equation}\label{descent_ineq_alg_2}
			\begin{split}
				f(x_{k+1})-f(x_{k})
				&\le \alpha_{k}\nabla f(x_{k})^{\top}d_{k}
				+\frac{e}{2}a_{k}\|d_{k}\|^{2}\alpha_{k}
				\left(\frac{-\,\nabla f(x_{k})^{\top}d_{k}}
				{a_{k}\|d_{k}\|^{2}e}\right)\\
				&=\frac{\alpha_{k}}{2}\,\nabla f(x_{k})^{\top}d_{k}.
			\end{split}
		\end{equation}
		Finally, substituting $\alpha_{k}$ from 
		\eqref{step_size_def} into \eqref{descent_ineq_alg_2} 
		completes the proof.
	\end{proof}
	
	\subsection{Linear convergence under strongly convex sets}
	
	We now establish the linear convergence rate of Algorithm~\ref{alg:fw_l0l1}.  
	For this result, the objective function $f$ is not required to be strongly convex; instead, we assume that the feasible set $Q$ is strongly convex.
	
	\begin{definition}[Strongly convex sets~\cite{levitin1966constrained}]\label{def:strongly_conv_sets}
		A set $Q$ is called \emph{strongly convex} if there exists $\lambda > 0$ such that for any $x, y \in Q$, every point $z$ satisfying
		\[
		\| z - (x+y)/2 \| \leq \lambda \| x-y \|^2
		\]
		also belongs to $Q$.
	\end{definition}
	
	Examples of strongly convex sets include $\ell_2$-balls and ellipsoids.  
	Such sets possess several important properties, one of which is the so-called \emph{scaling condition}.
	
	\begin{proposition}[Scaling condition for strongly convex sets]\label{def:scaling_condition}
		Let $Q$ be a strongly convex set with constant $\lambda$ (Definition~\ref{def:strongly_conv_sets}), and let $\psi$ be any nonzero vector.  
		Define $s = \operatorname*{argmin}_{y \in Q} \psi^\top y$.  
		Then, for all $x \in Q$, the following inequality holds:
		\[
		- \psi^\top (s - x) \geq 2 \lambda \| \psi \| \| s - x \|^2.
		\]
	\end{proposition}
	
	The original statement of this proposition appears in the seminal work of Levitin and Polyak~\cite{levitin1966constrained} (Theorem~6.1, p.~5).  
	An alternative proof can be found in~\cite{braun2022conditional} (Proposition~2.19), where the result is established for a more general class of $(\alpha, q)$-uniformly convex sets.
	
	\begin{theorem}\label{th:linear_strong_conv_set}
		Let $f$ be a convex $(L_0, L_1)$-smooth function such that $\|\nabla f(x)\| > c > 0$, and let $Q$ be a strongly convex set with constant $\lambda$.  
		Then Algorithm~\ref{alg:fw_l0l1} satisfies
		\begin{equation*}
			f(x_{k+1}) - f^*
			\leq \left(f(x_0) - f^*\right)
			\max\left\lbrace \frac{1}{2},\, 1 - \frac{\lambda}{2 e L_1} \right\rbrace ^T
			\max\left\lbrace \frac{1}{2},\, 1 - \frac{\lambda c}{2 e L_0} \right\rbrace^K,
		\end{equation*}
		where $T$ denotes the number of iterations such that $L_0 \leq L_1 \|\nabla f(x_k)\|$, and $K$ denotes the number of iterations such that $L_0 > L_1 \|\nabla f(x_k)\|$.
	\end{theorem}
	
	\begin{proof}
		We distinguish two cases.
		
		\textit{Case 1:} $\alpha_k < 1$.  
		Substituting the corresponding step size~\eqref{shortest_step_def_l0l1} into the descent inequality of Lemma~\ref{lemma:descent_lemma}, we obtain
		\begin{equation*}
			\begin{split}
				f(x_k) - f(x_{k+1})
				&\geq \frac{1}{2e}
				\cdot \frac{(\nabla f(x_k)^\top d_k)^2}{a_k \|d_k\|^2}
				\overset{\eqref{def:scaling_condition}}{\geq}
				\frac{\lambda}{e} \frac{\| \nabla f(x_k) \|}{a_k} (- \nabla f(x_k)^\top d_k)
				\overset{\eqref{fw_gap}}{\geq} \\
				&\frac{\lambda}{e} \frac{ \| \nabla f(x_k) \|}{a_k} (f(x_k) - f^*)
				= \frac{\lambda}{e} \cdot \frac{\| \nabla f(x_k) \|}{L_0 + L_1 \| \nabla f(x_k) \|} (f(x_k) - f^*).
			\end{split}
		\end{equation*}
		
		After straightforward algebraic manipulation, we get
		\begin{equation*}
			f(x_{k+1}) - f^* 
			\leq (f(x_k) - f^*) 
			\left( 1 - \frac{\lambda}{e} \cdot \frac{\| \nabla f(x_k) \|}{L_0 + L_1 \| \nabla f(x_k) \|} \right).
		\end{equation*}
		
		We now consider two subcases.  
		If $L_0 \leq L_1 \| \nabla f(x_k) \|$, then
		\begin{equation*}
			f(x_{k+1}) - f^*
			\leq (f(x_k) - f^*) 
			\left( 1 - \frac{\lambda}{2 e L_1} \right).
		\end{equation*}
		Otherwise, when $L_0 > L_1 \| \nabla f(x_k) \|$, we have
		\begin{equation*}
			f(x_{k+1}) - f^*
			\leq (f(x_k) - f^*) 
			\left( 1 - \frac{\lambda}{e} \cdot \frac{\| \nabla f(x_k) \|}{2 L_0} \right)
			\leq (f(x_k) - f^*)
			\left( 1 - \frac{\lambda c}{2 e L_0} \right).
		\end{equation*}
		
		\textit{Case 2:} $\alpha_k = 1$.  
		In this case, applying Lemma~\ref{lemma:descent_lemma} with $\alpha_k = 1$ yields
		\begin{equation*}
			f(x_{k+1}) - f(x_k)
			\leq \frac{\nabla f(x_k)^\top d_k}{2}
			\overset{\eqref{fw_gap}}{\leq}
			-\frac{1}{2}\bigl(f(x_k) - f^*\bigr),
		\end{equation*}
		which implies
		\begin{equation*}
			f(x_{k+1}) - f^*
			\leq \frac{1}{2}\bigl(f(x_k) - f^*\bigr).
		\end{equation*}
		
		Combining these inequalities gives the claimed linear convergence rate.
	\end{proof}
	
	As previously noted, the estimate in Theorem~\ref{th:linear_strong_conv_set} improves upon the classical Frank-Wolfe method with the shortest step size, whose rate satisfies
	\[
	f(x_k) - f^\ast \le (f(x_0) - f^\ast)
	\max\!\left\{\frac{1}{2},\, 1 - \frac{\lambda c}{2L}\right\}^{k-1},
	\]
	(see, e.g.,~\cite{braun2022conditional}), particularly in the regime where $L \gg L_1$ and $L \gg L_0$.
	
	\subsection{Linear Convergence under the Gradient Dominance Condition with $x^* \in \mathrm{Int}(Q)$}
	
	We now establish the linear convergence rate of Algorithm~\ref{alg:fw_l0l1} under a different set of assumptions.  
	In this setting, we consider objective functions that satisfy the Polyak--Łojasiewicz (PL) inequality.
	
	\begin{definition}[Polyak--Łojasiewicz condition]\label{def:pl_cond}
		A differentiable function $f$ is said to satisfy the \emph{PL condition} if there exists a constant $\mu > 0$ such that
		\[
		\frac{1}{2}\|\nabla f(x)\|^2 \geq \mu\,(f(x) - f^*), \quad \forall x.
		\]
	\end{definition}
	
	It is important to note that, in this case, the small parameter $\mu$ is absent when $L_0 > L_1 \|\nabla f(x_k)\|$, which results in faster convergence of our algorithm compared to the standard Frank--Wolfe method.
	
	The PL condition was introduced by Polyak~\cite{polyak1963gradient}.  
	This class of functions includes all strongly convex functions.  
	Later, Karimi et al.~\cite{karimi2016linear} showed that the PL condition also encompasses several broader families, such as weakly strongly convex functions, functions satisfying the restricted secant inequality, and essentially strongly convex functions.  
	Typical examples include logistic regression loss and least-squares objectives.
	
	In this setting, we relax the assumptions on the feasible set.  
	Specifically, we require only that $Q$ is convex and that the solution $x^*$ lies in its interior, i.e., $x^* \in \mathrm{Int}(Q)$.  
	We now recall a well-known scaling condition that holds under this assumption. Additionally recall that $B(x,r) \overset{\text{def}}{=} \{ z: \| z - x \| \leq r \}$ is the ball of radius $r$ around $x$.
	
	\begin{proposition}[Scaling condition for convex sets]\label{prop:scaling_condition_for_convex_set}
		Let $Q$ be a convex compact set, and let $f$ be a smooth convex function.  
		If there exists $r > 0$ such that $B(x,r) \subseteq Q$, then for all $x \in Q$ we have
		\[
		- \nabla f(x)^\top (s - x) \geq r\,\|\nabla f(x)\|.
		\]
	\end{proposition}
	
	The proof of this result can be found in~\cite[Proposition~2.16]{braun2022conditional}.
	
	We are now ready to prove the linear convergence rate for convex objectives satisfying the PL condition.  
	Notably, this result does not require a lower bound on $\|\nabla f(x_k)\|$, in contrast to Theorem~\ref{th:linear_strong_conv_set}.
	
	\begin{theorem}\label{th:linear_conv_pl_cond}
		Assume that $f$ is a convex $(L_0, L_1)$-smooth function satisfying the PL condition~\ref{def:pl_cond} with constant $\mu > 0$.  
		If there exists $r > 0$ such that $B(x^*, r) \subseteq Q$ for the solution $x^*$, then Algorithm~\ref{alg:fw_l0l1} satisfies
		\[
		f(x_{k+1}) - f^*
		\leq (f(x_0) - f^*)
		\max\!\left\{\frac{1}{2},\, 1 - \frac{r}{4 e L_1 D^2}\right\}^{T}
		\max\!\left\{\frac{1}{2},\, 1 - \frac{r^2 \mu}{2 e L_0 D^2}\right\}^{K},
		\]
		where $T$ denotes the number of iterations for which $L_0 \leq L_1 \|\nabla f(x_k)\|$,  
		and $K$ denotes the number of iterations for which $L_0 > L_1 \|\nabla f(x_k)\|$.
	\end{theorem}
	
	\begin{proof}
		As in Theorem~\ref{th:linear_strong_conv_set}, we consider two cases.
		
		\textit{Case 1:} Algorithm~\ref{alg:fw_l0l1} has $\alpha_k < 1$ at iteration $k$.  
		Starting from the descent inequality of Lemma~\ref{lemma:descent_lemma}, we have
		\begin{align*}
			f(x_k) - f(x_{k+1})
			&\geq \frac{1}{2e} \cdot 
			\frac{(\nabla f(x_k)^\top d_k)^2}{a_k \|d_k\|^2}
			\overset{\eqref{prop:scaling_condition_for_convex_set}}{\geq}
			\frac{r^2 \|\nabla f(x_k)\|^2}{2 e a_k \|d_k\|^2}.
		\end{align*}
		
		Recalling that $a_k = L_0 + L_1 \|\nabla f(x_k)\|$, we distinguish two subcases.
		
		When $L_0 > L_1 \|\nabla f(x_k)\|$, we obtain
		\begin{align*}
			f(x_k) - f(x_{k+1})
			&\geq \frac{r^2 \|\nabla f(x_k)\|^2}{4 e L_0 \|d_k\|^2}
			\overset{\eqref{def:pl_cond}}{\geq}
			\frac{r^2 \mu (f(x_k) - f^*)}{2 e L_0 \|d_k\|^2}
			\geq \frac{r^2 \mu (f(x_k) - f^*)}{2 e L_0 D^2}.
		\end{align*}
		
		After straightforward algebra, we obtain
		\begin{equation}\label{linear_rate_pl_alpha_less_1_case1}
			f(x_{k+1}) - f^*
			\leq (f(x_k) - f^*)
			\left( 1 - \frac{r^2 \mu}{2 e L_0 D^2} \right).
		\end{equation}
		
		In the remaining case, when $L_0 \leq L_1 \|\nabla f(x_k)\|$, we have
		\begin{equation*}
			\begin{split}
				f(x_k) - f(x_{k+1})
				&\overset{\eqref{prop:scaling_condition_for_convex_set}}{\geq}
				\frac{r\,(-\nabla f(x_k)^\top d_k)}{4 e L_1 \| d_k \|^2}
				\overset{\eqref{fw_gap}}{\ge}
				\frac{r (f(x_k) - f^*)}{4 e L_1 \| d_k \|^2}
				\geq \frac{r (f(x_k) - f^*)}{4 e L_1 D^2}.
			\end{split}
		\end{equation*}
		
		Consequently,
		\begin{equation}\label{linear_rate_pl_alpha_less_1_case2}
			f(x_{k+1}) - f^*
			\leq (f(x_k) - f^*)
			\left( 1 - \frac{r}{4 e L_1 D^2} \right).
		\end{equation}
		
		\textit{Case 2:} $\alpha_k = 1$.  
		Following the same reasoning as in Theorem~\ref{th:linear_strong_conv_set}, we use the descent inequality of Lemma~\ref{lemma:descent_lemma} with the appropriate step size:
		\[
		f(x_{k+1}) - f(x_k) \leq \frac{\nabla f(x_k)^\top d_k}{2} 
		\overset{\eqref{fw_gap}}{\leq} 
		-\frac{1}{2} \left(f(x_k) - f^*\right),
		\]
		which implies
		\[
		f(x_{k+1}) - f^* \leq \frac{1}{2} (f(x_k) - f^*).
		\]
		
		Combining these inequalities yields the claimed linear convergence rate.
	\end{proof}
	
	Finally, we compare this result with the classical Frank--Wolfe convergence estimate for $L$-smooth functions under the same setting:
	\[
	f(x_k) - f^\ast \le (f(x_0) - f^\ast)
	\max\!\left\{\frac{1}{2},\, 1-\frac{r^2\mu}{L D^2}\right\}^{\,k-1}.
	\]
	Once again, we observe an acceleration when $L \gg L_1$ and $L \gg L_0$.  
	However, we emphasize an additional distinction from the classical formulation: in our case, acceleration occurs not only due to the inequalities $L \gg L_1$ and $L \gg L_0$, but also due to the absence of the $\mu$ parameter from the PL condition~\ref{def:pl_cond} in the case $L_0 > L_1 \|\nabla f(x_k)\|$.  
	Since $\mu$ is typically small, the classical convergence estimate is often slower than the rate established for our proposed algorithm.
	
	\section{Adaptation Procedure for the $(L_0, L_1)$-FW Algorithm}
	
	The adaptive variant of Algorithm~\ref{alg:fw_l0l1} is presented in Listing~\ref{alg:fw_l0l1_adapt} (denoted as Adapt $(L_0, L_1)$-FW).
	
	As mentioned in the introduction, it employs a more flexible adaptation mechanism compared to standard approaches. In our method, the parameters $L_0$ and $L_1$ are adjusted individually, in proportion to their respective contributions to the total value $a_k = L_0^{(k)} + L_1^{(k)} \| \nabla f(x_k) \|$. Moreover, both parameters are divided simultaneously but multiplied sequentially, one after another. The corresponding adaptive mechanism is illustrated in lines~\ref{divide_parameter_l0},~\ref{divide_parameter_l1},~\ref{multiply_parameter_l0}, and~\ref{multiply_parameter_l1} of Algorithm~\ref{alg:fw_l0l1_adapt}.
	
	\begin{algorithm}
		\caption{Adaptive $(L_0, L_1)$-Frank-Wolfe Algorithm (Adapt $(L_0,L_1)$-FW)}
		\label{alg:fw_l0l1_adapt}
		\begin{algorithmic}[1]
			\State \textbf{Input:} initial parameters $L_0^{(0)} > 0$, $L_1^{(0)} > 0$, maximum number of iterations $N$, scaling factor $\rho > 2$
			\State toggle $\gets 0$
			\For{$k = 0, 1, \dots, N-1$}
			\State $s_k \gets \mathrm{LMO}(\nabla f(x_k))$
			\State $d_k \gets s_k - x_k$
			\State $a_k \gets L_0^{(k)} + L_1^{(k)} \|\nabla f(x_k)\|$
			
			\State $L_0^{(k)} \gets L_0^{(k)} / (\rho + L_0^{(k)} / a_k)$ \label{divide_parameter_l0}
			\State $L_1^{(k)} \gets L_1^{(k)} / (\rho + (L_1^{(k)} \|\nabla f(x_k)\|) / a_k)$ \label{divide_parameter_l1}
			
			\While{True}
			\State $a_k \gets L_0^{(k)} + L_1^{(k)} \|\nabla f(x_k)\|$
			\State $\displaystyle 
			\alpha(k) \gets \min \left\{ 1, 
			\frac{- \nabla f(x_k)^\top d_k}{a_k \|d_k\|^2 e} \right\}$
			\State $x_{k+1} \gets x_k + \alpha_k d_k$
			
			\If{$f(x_{k+1}) \le f(x_k) + \alpha(k) \nabla f(x_k)^\top d_k + 
				\frac{a_k e}{2} \alpha(k)^2 \|d_k\|^2$} \label{adapt_checks}
			\State \textbf{break}
			\Else
			\If{toggle $= 0$}
			\State $L_0^{(k)} \gets L_0^{(k)} \cdot (\rho - L_0^{(k)} / a_k)$ \label{multiply_parameter_l0}
			\State toggle $\gets 1$
			\Else
			\State $L_1^{(k)} \gets L_1^{(k)} \cdot (\rho - (L_1^{(k)} \|\nabla f(x_k)\|) / a_k)$ \label{multiply_parameter_l1}
			\State toggle $\gets 0$
			\EndIf
			\EndIf
			\EndWhile
			\EndFor
			\State \textbf{Output:} final point $x_N$
		\end{algorithmic}
	\end{algorithm}
	
	We now show that the adaptive procedure in Algorithm~\ref{alg:fw_l0l1_adapt} does not deteriorate the convergence rate, and that all convergence theorems preserve their corresponding estimates.  
	Our analysis follows the approach of~\cite{nesterov2013gradient} (Lemma~4), adapted to our setting.
	
	\begin{lemma}
		Assume that $f$ is an $(L_0,L_1)$-smooth function~\ref{def:l0_l1} with constants $L_0$ and $L_1$, and let $n_i$ denote the number of inequality checks in line~\ref{adapt_checks} of Algorithm~\ref{alg:fw_l0l1_adapt}.  
		Then, for an adaptation factor $\rho > 2$, the following bound holds:
		\begin{align}
			\sum_{i = 0}^{N} n_i  & \le 
			N \!\left(1 + \frac{\log(\rho + 1)}{\log(\rho - 1)}\right)
			+ \frac{1}{\log(\rho - 1)} 
			\log \frac{ 2 \rho \max\{L_0, L_1\} }{L_0^{(0)} + L_1^{(0)}} \notag \\[2pt]
			&= O(N).
		\end{align}
	\end{lemma}
	
	\begin{proof}
		From Algorithm~\ref{alg:fw_l0l1_adapt} and the fact that $L_0^{(k)} / a_k \leq 1$ and $(L_1^{(k)} \|\nabla f(x_k)\|) / a_k \leq 1$ (see the definition of the $a_k$ in the Algorithm~\ref{alg:fw_l0l1_adapt}), during iteration $k$ we have
		\[
		L_0^{(k)} + L_1^{(k)} 
		\ge 
		\frac{(\rho - 1)^{n_k - 1}}{\rho + 1}
		\bigl(L_0^{(k-1)} + L_1^{(k-1)}\bigr).
		\]
		Taking natural logarithms yields
		\[
		n_k 
		\le 
		1 
		+ \frac{\log(\rho + 1)}{\log(\rho - 1)}
		+ \frac{1}{\log(\rho - 1)}
		\log\frac{L_0^{(k)} + L_1^{(k)}}{L_0^{(k-1)} + L_1^{(k-1)}}.
		\]
		Summing over $k = 0, \dots, N$
		\[
		\begin{aligned}
			\sum_{k=0}^{N} n_k 
			&\le 
			N 
			+ N \frac{\log(\rho + 1)}{\log(\rho - 1)}
			+ \frac{1}{\log(\rho - 1)}
			\log\frac{L_0^{(N)} + L_1^{(N)}}{L_0^{(0)} + L_1^{(0)}} 
			\\
			&\leq 
			N \!\left(1 + \frac{\log(\rho + 1)}{\log(\rho - 1)}\right)
			+ \frac{1}{\log(\rho - 1)} 
			\log \frac{ 2 \rho \max\{L_0, L_1\} }{L_0^{(0)} + L_1^{(0)}}
		\end{aligned}
		\]
		gives the claimed $O(N)$ bound.
	\end{proof}
	
	\begin{lemma}[Descent lemma for adaptive parameters]\label{lemma:descent_lemma_adapt}
		Let $f$ be a convex function satisfying $(L_0,L_1)$-smoothness~\ref{def:l0_l1}.  
		Then, for Algorithm~\ref{alg:fw_l0l1_adapt} and for each iteration $k \ge 1$, we have
		\begin{equation*}
			f(x_{k+1}) - f(x_k) \leq \frac{\nabla f(x_k)^\top d_k}{2} \cdot \min \left\{ 1,\;
			\frac{-\, \nabla f(x_k)^\top d_k}{\left(L_0^{(k)} + L_1^{(k)} \|\nabla f(x_k)\|\right) \|d_k\|^2 e} \right\}.
		\end{equation*}
	\end{lemma}
	
	\begin{proof}
		This follows directly from the fact that, at each iteration, the parameters $L_0^{(k)}$ and $L_1^{(k)}$ are chosen to satisfy inequality~\eqref{th:sublinear_conv_rate2:upper_bnd}.
	\end{proof}
	
	\begin{corollary}
		If $f$ satisfies all assumptions of Theorem~\ref{th:linear_strong_conv_set}, then Algorithm~\ref{alg:fw_l0l1_adapt} satisfies
		\begin{equation*}
			f(x_{k+1}) - f^*
			\leq (f(x_0) - f^*) 
			\max\left\lbrace {\frac{1}{2}}, 1 - \frac{\lambda}{2 e L_1^{\max}} \right\rbrace ^T
			\max\left\lbrace \frac{1}{2}, 1 - \frac{\lambda c}{2 e L_0^{\max}} \right\rbrace^K,
		\end{equation*}
		where $T$ and $K$ are defined as in Theorem~\ref{th:linear_strong_conv_set}, and $L_j^{\max} = \max_{i = 0, \dots, k} L_j^{(i)}$.
	\end{corollary}
	
	\begin{corollary}
		If $f$ satisfies all assumptions of Theorem~\ref{th:linear_conv_pl_cond}, then Algorithm~\ref{alg:fw_l0l1_adapt} satisfies
		\begin{equation*}
			f(x_{k+1}) - f^*
			\leq (f(x_0) - f^*)
			\max\left\lbrace \frac{1}{2}, 1 - \frac{r}{4 e L_1^{\max} D^2} \right\rbrace ^T
			\max\left\lbrace \frac{1}{2}, 1 - \frac{r^2 \mu}{2 e L_0^{\max} D^2} \right\rbrace ^K,
		\end{equation*}
		where $T$, $K$, and $L_j^{\max}$ are defined as above.
	\end{corollary}
	
	\section{Convergence rate for the general convex case}
	
	In this section we prove the suboptimal convergence rate of Algorithm~\ref{alg:fw_l0l1} for the general convex case.
	
	\begin{lemma}\label{lemma:subopt_rate}
		Assume that for a sequence $0 < h_0, \ldots, h_N$ and some $K_{\text{max}} > 0$ we have
		\begin{equation*}
			h_{k+1} \leq h_k \left( 1 - \frac{h_k}{K_{\text{max}}} \right).
		\end{equation*}
		Then for $k \geq 0 \in \mathbb{Z}$
		\begin{equation*}
			h_{k+1} \leq \frac{K_{\text{max}}}{k + 3}.
		\end{equation*}
	\end{lemma}
	
	\begin{proof}
		The proof proceeds by induction. For $k = 0$, we have
		\begin{equation*}
			h_1 \leq h_0 \left( 1 - \frac{h_0}{K_{\text{max}}} \right)
			\leq \frac{K_{\text{max}}}{4} \leq \frac{K_{\text{max}}}{3} = \frac{K_{\text{max}}}{0 + 3},
		\end{equation*}
		where the inequality follows from the upper bound of the function $f(x) = x - \frac{x^2}{K_{\text{max}}}$ for $x \geq 0$.
		
		Assume the statement holds for some $k$, and let us prove
		\begin{equation*}
			h_{k+2} \leq \frac{K_{\text{max}}}{k + 4}.
		\end{equation*}
		
		We consider two cases.  
		If $h_{k+1} \leq \frac{K_{\text{max}}}{k + 4}$, the inequality follows immediately.  
		Otherwise, if $h_{k+1} > \frac{K_{\text{max}}}{k + 4}$, then
		\begin{equation*}
			h_{k+2} \leq h_{k+1} \frac{k + 3}{k + 4} \leq \frac{K_{\text{max}}}{k + 3} \frac{k + 3}{k + 4} = \frac{K_{\text{max}}}{k + 4},
		\end{equation*}
		where the second inequality follows from the induction hypothesis.
	\end{proof}
	
	\begin{theorem}\label{th:sublinear_conv_rate2}
		Assume that $f$ is a convex function satisfying the $(L_0,L_1)$-smoothness condition~\ref{def:l0_l1}. Then, for Algorithm~\ref{alg:fw_l0l1}, for each iteration $k \geq 1$, we have
		\begin{equation*}
			f(x^k) - f(x^*) \leq \frac{2eD^2 \left( L_0 + L_1 \max_k \| \nabla f(x_k) \| \right)}{k + 3},
		\end{equation*}
		where $D = \max_{x,y \in Q} \|x - y\|$.
	\end{theorem}
	
	\begin{proof}
		We use the descent inequality from Lemma~\ref{lemma:descent_lemma}:
		\begin{equation}\label{descent_ineq_sublinear}
			f(x_{k+1}) - f(x_k)
			\leq
			\frac{\nabla f(x_k)^\top d_k}{2}
			\cdot
			\min \left\{
			1,\;
			\frac{-\, \nabla f(x_k)^\top d_k}{
				\left(L_0 + L_1 \|\nabla f(x_k)\|\right)
				\|d_k\|^2 e}
			\right\}.
		\end{equation}
		
		We consider two cases.
		
		\medskip
		\noindent
		\textbf{Case 1:} $\boldsymbol{\alpha(k) = 1.}$  
		From~\eqref{descent_ineq_sublinear}, we have
		\begin{equation*}
			f(x_{k+1}) - f(x_k)
			\leq
			\frac{1}{2} \nabla f(x_k)^\top d_k
			\overset{\eqref{fw_gap}}{\leq}
			\frac{1}{2} (f^* - f(x_k)),
		\end{equation*}
		and hence
		\begin{equation*}
			f(x_{k+1}) - f^* \leq \frac{1}{2} (f(x_k) - f^*).
		\end{equation*}
		
		Subtracting $f^*$ from both sides of
		\eqref{th:sublinear_conv_rate2:upper_bnd_only_e} with $\alpha(k)=1$ yields
		\begin{equation*}
			\begin{split}
				f(x_{k+1}) - f^*
				&\leq -f^* + f(x_k) + \nabla f(x_k)^\top d_k
				+ \frac{a_k}{2} \|d_k\|^2 e \\
				&\overset{\eqref{fw_step_size}}{\leq}
				\underbrace{-f^* + f(x_k)
					+ \nabla f(x_k)^\top (x^* - x_k)}_{\leq 0 \text{ (by convexity)}}
				+ \frac{a_k}{2} \|d_k\|^2 e
				\leq \frac{a_k}{2} \|d_k\|^2 e.
			\end{split}
		\end{equation*}
		
		Thus, for $\alpha(k) = 1$, we obtain
		\begin{equation}\label{alpha_1_case_desc_ineq}
			f(x_{k+1}) - f^* \leq
			\min \left\{
			\frac{a_k}{2} \|d_k\|^2 e,\;
			\frac{1}{2}(f(x_k) - f^*)
			\right\}.
		\end{equation}
		
		\medskip
		\noindent
		\textbf{Case 2:} $\boldsymbol{\alpha(k) < 1.}$  
		Substituting the step size~\eqref{step_size_def} into
		\eqref{descent_ineq_sublinear} yields
		\begin{equation}\label{descent_ineq_alg_2_step_size_less_1}
			f(x_{k+1}) - f(x_k)
			\leq
			\frac{1}{2}
			\frac{(\nabla f(x_k)^\top d_k)^2}{
				a_k \|d_k\|^2 e}.
		\end{equation}
		
		We now establish the convergence rate by induction.
		For $k=0$ and $\alpha(0)=1$, from~\eqref{alpha_1_case_desc_ineq} we have
		\[
		f(x_{1}) - f^*
		\leq \frac{a_0}{2} \|d_0\|^2 e
		\leq \frac{2 \max \{a_0, a_1\} D^2 e}{4}.
		\]
		For arbitrary $k + 1$, the desired inequality follows directly from the induction hypothesis and the right-hand side of~\eqref{alpha_1_case_desc_ineq}.
		
		If $\alpha(k) < 1$, we slightly reformulate~\eqref{descent_ineq_alg_2}:
		\[
		f(x_{k+1}) - f^*
		\leq (f(x_k) - f^*)
		- \frac{1}{2}
		\frac{(\nabla f(x_k)^\top d_k)^2}{
			a_k \|d_k\|^2 e}
		\overset{\eqref{fw_gap}}{\leq}
		(f(x_k) - f^*)
		\left(
		1 - \frac{\nabla f(x_k)^\top d_k}{
			2 a_k \|d_k\|^2 e}
		\right).
		\]
		
		Denoting $h_k = f(x_k) - f^*$ and $K = 2 a_k \|d_k\|^2 e$ (note that $h_k \leq K$ by~\eqref{descent_ineq_sublinear} and the convexity of $f$), and applying Lemma~\ref{lemma:subopt_rate}, we obtain
		\[
		f(x_{k+1}) - f^* \leq
		\frac{2eD^2 \left( L_0 + L_1 \max_k \|\nabla f(x_k)\| \right)}{k + 3}.
		\]
	\end{proof}
	
	It is well-known that the $O(1/k)$ rate is optimal for the Frank--Wolfe method~\cite{polyak1983introduction}.  
	Hence, asymptotically, our estimate matches the standard Frank--Wolfe rate.

	\section{Numerical Experiments}\label{sec_numerical_experiments}
	
	In this section, we present numerical experiments illustrating the performance of the proposed Algorithm~\ref{alg:fw_l0l1} ($(L_0,L_1)$-FW) and its adaptive variant, Algorithm~\ref{alg:fw_l0l1_adapt} (Adapt $(L_0,L_1)$-FW).  
	The aim of these experiments is to demonstrate the advantages of the proposed step-size selection rule compared with classical and previously known adaptive Frank--Wolfe variants.
	
	\subsection{Objective Function}
	We consider the optimization of the logistic regression objective
	\[
	f(x) = \frac{1}{n} \sum_{i=1}^n \log\big(1 + \exp(-y_i (A x)_i)\big),
	\]
	where \(x \in \mathbb{R}^d\). 
	
	Each component function $f_i(x) = \log\bigl(1 + \exp(-y_i a_i^\top x)\bigr)$
	is classically smooth with Lipschitz gradient constant \(L_i = \|a_i\|^2\). Moreover, it satisfies the \((L_0,L_1)\)-smoothness condition with parameters \(L_0 = 0\) and \(L_1 = \|a_i\|\); see~\cite{gorbunov2024methods}, typically with \(\|a\| \gg 1\).
	
	We note that the class of \((L_0,L_1)\)-smooth functions is generally not closed under finite summation. Nevertheless, for the implementation of methods with fixed step sizes, we use the conservative global estimates
	$L = \max_{1 \leq i \leq n} \|a_i\|^2$, $L_0 = 0$, $L_1 = \max_{1 \leq i \leq n} \|a_i\|$. In contrast, adaptive step-size strategies do not require prior knowledge of these parameters. 
	
	\subsection{Data Generation}
	The matrix \(A\) and vector \(y\) were generated using standard normal distributions:
	\[
	A \sim \mathcal{N}(0, 1).
	\]
	Each column of \(A\) was multiplied by a scalar \(2^j\), where \(j\) is the column index.  
	The label vector \(y\) was then generated as
	\[
	y = \mathrm{sign}(A^\top x_{\mathrm{sol}} + \mathrm{noise}),
	\]
	where \(x_{\mathrm{sol}}\) is a randomly generated solution vector, and \(\mathrm{sign}\) returns the sign of each element.
	
	In all cases, the data are uncorrelated and slightly noisy.  
	We considered three geometric settings for the data distribution:
	
	\begin{enumerate}
		\item Random points uniformly distributed on the $\ell_2$-ball of radius 25, with increasing number of features and dimension.  
		\item Random points uniformly distributed on the 1-simplex, with increasing dimension.  
		\item Random points uniformly distributed on the $\ell_\infty$-ball, with increasing dimension.
	\end{enumerate}
	
	\subsection{Compared Algorithms}
	The following algorithms were compared:
	
	\begin{itemize}
		\item \textbf{Classic FW}: the standard Frank--Wolfe algorithm using the shortest-step rule
		\[
		\alpha_k = \min\left\{ \frac{\nabla f(x_k)^\top (s_k - x_k)}{L \| s_k - x_k \|^2},\, 1 \right\},
		\]
		where \(L\) is the Lipschitz constant of the gradient and \(s_k\) is the solution of the linear minimization oracle (LMO).
		
		\item \textbf{Adaptive Classic FW}: identical to the classical Frank--Wolfe algorithm, but with the smoothness constant \(L\) updated adaptively at each iteration, as analyzed in~\cite{aivazian2023adaptive}.
		
		\item \textbf{$(L_0,L_1)$-FW}: proposed algorithm (Algorithm~\ref{alg:fw_l0l1}), with fixed parameters \((L_0, L_1)\) set in advance, e.g., known for the logistic regression objective.
		
		\item \textbf{Adapt $(L_0,L_1)$-FW}: proposed algorithm (Algorithm~\ref{alg:fw_l0l1_adapt}), which employs an adaptive step-size rule based on the constants \((L_0, L_1)\), with sequential updates for each parameter.
	\end{itemize}
	
	\subsection{Results and Discussion}
	
	The experimental results demonstrate that the additional adaptive correction in the proposed Adapt $(L_0,L_1)$-FW algorithm consistently accelerates convergence across all test cases.  
	This algorithm achieved the best performance in all experiments. Similarly, the $(L_0,L_1)$-FW algorithm also significantly outperforms the standard Frank--Wolfe algorithm.
	
	The improvement provided by the adaptive procedure stems from the finer tuning of the step size: the adaptive update in Adapt $(L_0,L_1)$-FW increases the parameters more smoothly and conservatively, avoiding abrupt changes that could hinder convergence.  
	Additionally, alternating the parameter updates enhances the convergence speed, as the aggregate value
	\[
	a_k = L_0 + L_1 \| \nabla f(x_k) \|
	\]
	increases more gradually and stabilizes earlier than in more aggressive adaptive schemes.
	
	\begin{figure}[h!]
		\centering
		\includegraphics[width=\textwidth]{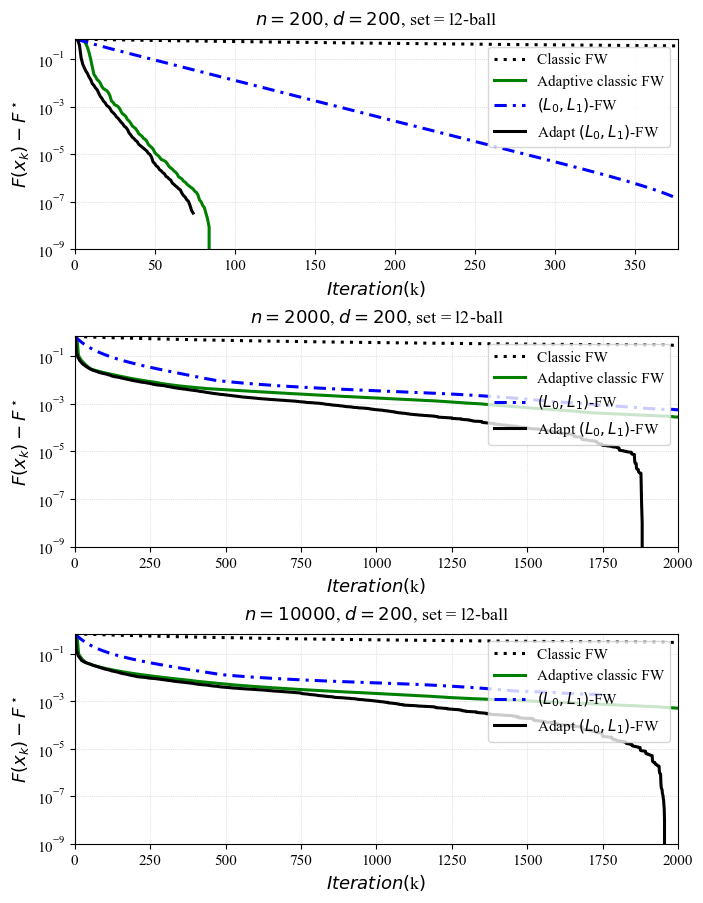}
		\caption{$\ell_2$-ball set: increasing number of data points (log scale on the Y-axis).}
	\end{figure}
	
	\begin{figure}[h!]
		\centering
		\includegraphics[width=\textwidth]{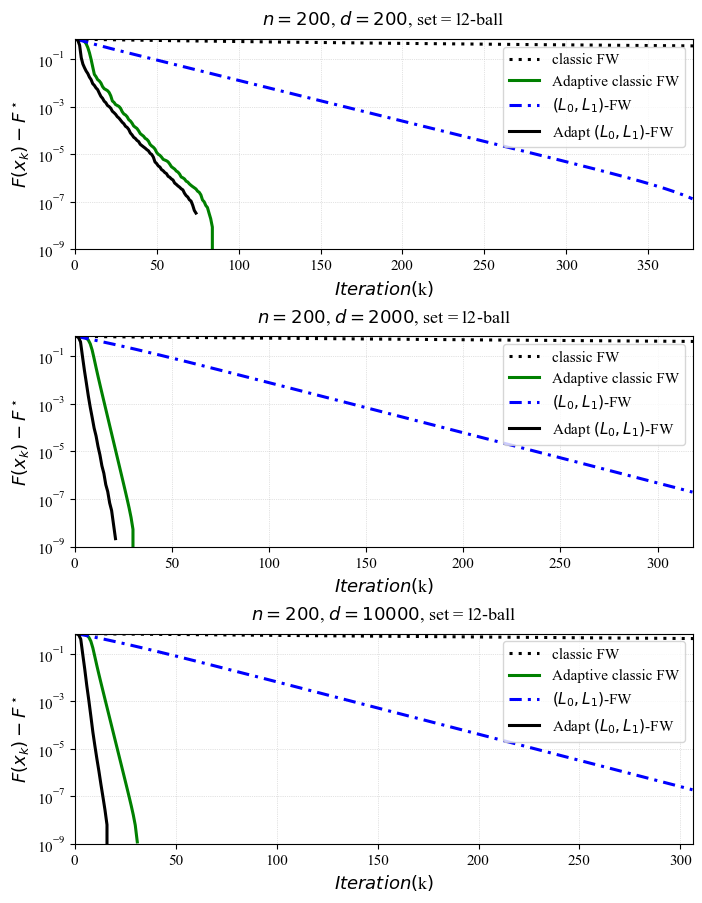}
		\caption{$\ell_2$-ball set: increasing dimension (log scale on the Y-axis).}
	\end{figure}
	
	\begin{figure}[h!]
		\centering
		\includegraphics[width=\textwidth]{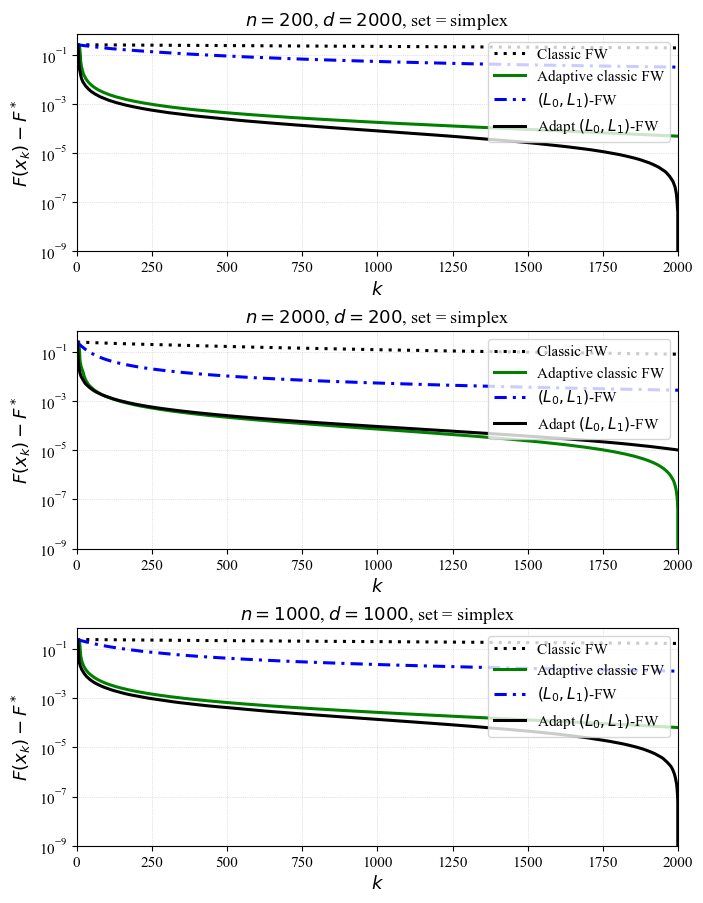}
		\caption{1-simplex set: increasing dimension (log scale on the Y-axis).}
	\end{figure}
	
	\begin{figure}[h!]
		\centering
		\includegraphics[width=\textwidth]{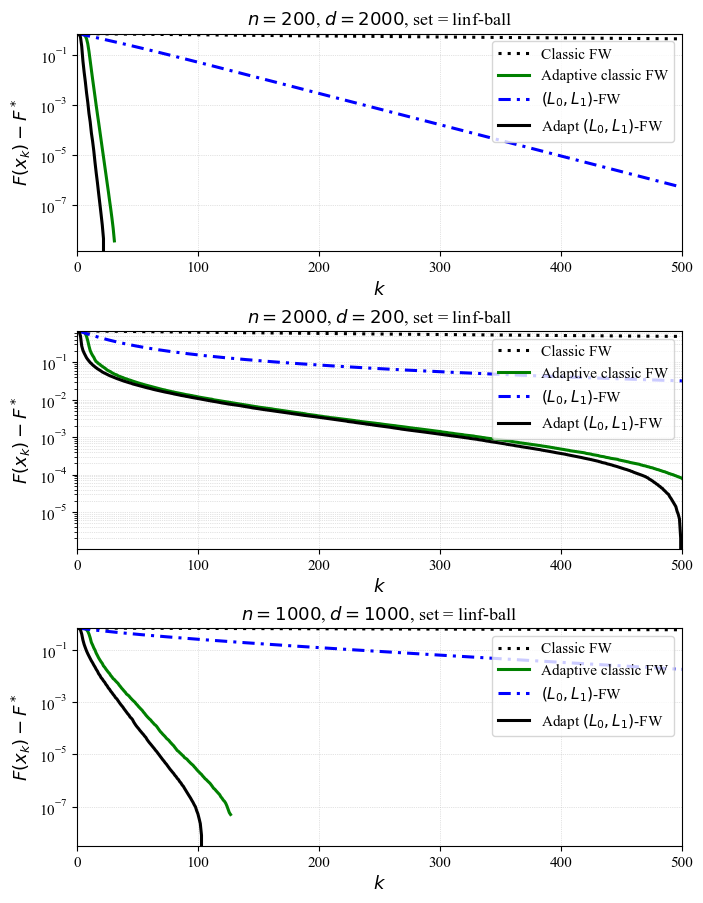}
		\caption{$\ell_\infty$-ball set: increasing dimension (log scale on the Y-axis).}
	\end{figure}
	
	\section{Conclusion}
	
	In this work, we introduced two new versions of the Frank--Wolfe method for $(L_0,L_1)$-smooth objective functions: the \emph{$(L_0,L_1)$-FW} algorithm and its adaptive counterpart, the \emph{Adaptive $(L_0,L_1)$-FW} algorithm.  
	For both methods, we established convergence guarantees under multiple settings, including general convex objectives, convex objectives over strongly convex feasible sets, and convex objectives satisfying the Polyak--Łojasiewicz (PL) condition with the solution in the interior of the feasible set.
	
	The numerical experiments confirm the theoretical results and show that the proposed algorithms outperform the classical Frank--Wolfe method and existing adaptive variants.  
	In particular, the adaptive version exhibits a clear advantage due to its capability to dynamically adjust the parameters $L_0$ and $L_1$ based on the local problem geometry.
	
	Our adaptation strategy differs from previous approaches that adjust only a single Lipschitz-like constant.  
	Instead, we adapt both parameters simultaneously, accounting for their relative contributions in the composite smoothness model, which renders the procedure more flexible and responsive.
	
	Despite these encouraging results, several open questions remain. For instance, how large can the adaptive parameters become relative to the true $(L_0,L_1)$ values? And what are the optimal initial choices of $L_0$ and $L_1$ when these parameters are not known a priori?
	
	\noindent{\bf Acknowledgments}
	
	\noindent  ARC 09-2025 001
	This work was supported by the Ministry of Science and Higher Education of the Russian Federation within state assignment no. 075-03-2024-074 under the project “Study of Asymptotic Characteristics of Fluctuations of Differential Equations and Systems, and Optimization Methods.”
	
	\bibliographystyle{splncs04}
	\bibliography{FW_L0L1}
	
\end{document}